\documentclass[11pt,english]{smfart}

\usepackage[T1]{fontenc}
\usepackage[english,francais]{babel}
\usepackage{amssymb,url,xspace,smfthm}

\makeatletter
    \def\ps@copyright{\ps@empty
    \def\@oddfoot{\hfil\small\copyright 2012, \SMF}}
\makeatother

\newcommand{\SMF}{Soci\'et\'e Ma\-th\'e\-Ma\-ti\-que de France}
\newcommand{\BibTeX}{{\scshape Bib}\kern-.08em\TeX}
\newcommand{\T}{\S\kern .15em\relax }
\newcommand{\AMS}{$\mathcal{A}$\kern-.1667em\lower.5ex\hbox
        {$\mathcal{M}$}\kern-.125em$\mathcal{S}$}

\usepackage[all]{xy}
\input yoga.sty
\makeatletter
    \let\markboth\org@markboth
\makeatother
\renewcommand{\k}{{\mathbf k}}

\DeclareMathOperator{\im}{im}
\DeclareMathOperator{\hull}{hull}
\DeclareMathOperator{\diag}{diag}
\DeclareMathOperator{\Diag}{Diag}
\DeclareMathOperator{\rank}{rank}
\DeclareMathOperator{\hgt}{ht}

\newcommand{\Z}{\ZZ}
\newcommand{\Q}{\QQ}
\newcommand{\F}{\FF}

\newcommand{\m}{\mathfrak m}

\newcommand{\GD}{{$G\Theta$}}
\newcommand{\AG}{{$AG$}}
\newcommand{\CG}{{$CG$}}
\newcommand{\AGD}{{$AG\Theta$}}
\newcommand{\CGD}{{$CG\Theta$}}
\newcommand{\CGL}{{$CG\Lambda$}}
\newcommand{\cO}{\mathcal O}
\newcommand{\cM}{\mathcal M}
\newcommand{\cS}{\mathcal S}
\newcommand{\cE}{\mathcal E}
\newcommand{\cF}{\mathcal F}
\newcommand{\cG}{\mathcal G}
\newcommand{\cQ}{\mathcal Q}
\newcommand{\cHom}{\mathcal Hom}
\newcommand{\gl}{\mathop{\mathfrak{gl}}\nolimits}
\newcommand{\even}{{\mathrm{even}}}
\newcommand{\h}{H}
\newcommand{\GGa}{{\mathbb G_a}}

\usepackage{verbatim}

\newtheorem{Proposition}{Proposition}[section]
\newtheorem{Theorem}[Proposition]{Theorem}
\newtheorem{Lemma}[Proposition]{Lemma}
\newtheorem{Corollary}[Proposition]{Corollary}

\theoremstyle{definition}

\newtheorem{Remark}[Proposition]{Remark}
\newtheorem{Definition}[Proposition]{Definition}
\newtheorem{Notation}[Proposition]{Notation}

\title[Grosshans Filtration]{Good Grosshans filtration in a family}

\date{\today}
\author{Wilberd van der Kallen}

\address{Mathematisch Instituut Universiteit Utrecht\\
P.O. Box 80.010, NL-3508 TA Utrecht, The Netherlands}
\email{W.vanderKallen@uu.nl}
\urladdr{http://www.xs4all.nl/~wilberdk/}

\begin{document}

\def\smfbyname{}

\begin{abstract}We reprove the main result of our joint work \cite{Srinivas vdK}, with the base field replaced by a
commutative noetherian ring $\k$.
 This has repercussions for the cohomology $H^*(G,A)$
of a reductive group scheme $G$ over $\k$, with coefficients in a finitely generated commutative $\k$-algebra $A$.
 For clarity we follow  \cite{Srinivas vdK} closely.
\end{abstract}

\begin{altabstract}
Nous g\'en\'eralisons le r\'esultat principal de  \cite{Srinivas vdK}, en rempla\c cent le corps de base par un anneau 
commutatif noeth\'erien $\k$. Ainsi on obtient de l'information sur la cohomologie $H^*(G,A)$,
 o\`u $G$ est un sch\'ema en groupes
r\'eductif sur $\k$ et $A$ est une $\k$-alg\`ebre de type fini. 
Nous suivons les grandes lignes du texte original \cite{Srinivas vdK}. 
\end{altabstract}

\maketitle

\section{Introduction}Let $\k$ be a noetherian ring. 
Consider  a flat linear algebraic  group scheme
$G$ defined over  $\k$.
Recall that $G$ has the cohomological finite generation property (CFG)
if the following holds:
Let $A$ be a finitely generated commutative $\k$-algebra on which $G$ acts
 rationally by $\k$-algebra automorphisms. (So $G$ acts from the right
 on $\Spec(A)$.)
Then the cohomology ring $H^*(G,A)$ is finitely generated
as a $\k$-algebra. Here, as in
\cite[I.4]{Jantzen}, we use the cohomology introduced by
Hochschild, also known as `rational cohomology'.

This note is part of the project of studying (CFG) for reductive $G$.
More specifically, the intent of this note is to generalize the main result of \cite{Srinivas vdK}
to the case where the base ring of $\GL_N$  is 
our noetherian ring $\k$.
That will allow to enlarge the scope
of several results
in \cite{TvdK}, \cite{FvdK}. 
Let us give an example. 
Let $G$ be a reductive group scheme over $\Spec(\k)$ in the sense of SGA3. Recall this means
that $G$ is affine and smooth over $\Spec(\k)$ with geometric fibers that are connected reductive. 
Let $G$ act rationally by 
$\k$-algebra automorphisms on a finitely generated commutative $\k$-algebra $A$.
We do not know (CFG) in this generality, but now we can state at least that the $H^m(G,A)$ are noetherian modules
over the ring of invariants $A^G$. And if $\k$ contains a finite ring we do indeed know that $H^*(G,A)$ is a finitely
generated $\k$-algebra.
See  section \ref{consequences} for these results and related material.

To formulate the main result, let $N\geq1$ and let $G$ be the affine algebraic group 
$\GL_N$ or $\SL_N$ over  $\k$. We use notations and terminology as in \cite{Srinivas vdK}, \cite{FvdK}.
Recall in particular that a $G$-module $V$ module is said to have  \emph{good Grosshans filtration} if the embedding 
$\gr V\to \hull_\nabla(\gr V)$ of Grosshans is an isomorphism \cite[Definition 27]{FvdK}. Such a module is $G$-acyclic.
It does not need to be flat over $\k$. The module $V$ has a good Grosshans filtration if and only if it satisfies the
following cohomological criterion: $H^i(G,V\otimes_\k \nabla(\lambda))$ vanishes for all $i>0$ and all dominant weights
$\lambda$. Over fields this is the familiar criterion for having a good filtration. Indeed over a field there is no difference 
between `good filtration' and `good Grosshans filtration'. But modules with good filtration are required to be free over
$\k$ and this is not the right requirement in our present setting. We wish to allow the filtration of $V$ to have
an associated graded
that is a direct sum of modules of the form $\nabla(\lambda)\otimes_\k  J(\lambda)$ with $G$ acting trivially on $J(\lambda)$.
The $J(\lambda)$ do not have to be free over $\k$; they even do not have to be flat over $\k$.

Let $A$ be a finitely generated commutative $\k$-algebra on which $G$ acts
rationally by $\k$-algebra automorphisms. Let $M$ be a noetherian $A$-module
on which $G$ acts compatibly. This means that the structure map $A\otimes_\k  M\to
M$ is a $G$-module map. We also say that $M$ is a (noetherian) $AG$-module.
(Later our convention will be that any $AG$-module is  noetherian.)

Our main theorem is

\begin{Theorem}\label{maingood}
If $A$ has a good Grosshans filtration, then
there is a finite resolution $$0\to M\to N_0\to N_1\to\cdots\to N_d\to 0$$
where the $N_i$ are noetherian $AG$-modules with good Grosshans filtration.
\end{Theorem}

\begin{Corollary}\label{noetheriancohomology}
 The $H^i(G,M)$ are noetherian $A^G$-modules and they vanish for $i\gg0$.
\end{Corollary}
\begin{proof}
One may compute $H^*(G,M)$ with the resolution $N_0\to \cdots N_d\to0$. So the result follows from invariant
theory \cite[Theorem 12, Theorem 9]{FvdK}.
\end{proof}

\begin{Remark}
 It is natural to ask if the same results hold for other Dynkin types. For the Corollary the answer is yes, because
of Theorem \ref{bounded power fg:theorem} below. For Theorem \ref{maingood} we do not know how to keep the $N_i$ noetherian,
but otherwise it goes through by \cite[Proposition 28]{FvdK} and Theorem \ref{bounded power fg:theorem} below.
\end{Remark}

We will actually prove a more technical version of the theorem. This is the key difference with the proof
in \cite{Srinivas vdK}.
Recall that the fundamental weights $\varpi_1$, \dots , $\varpi_N$
are given by $\varpi_i=\sum_{j=1}^i\epsilon_j$. Let $\rho$ be their sum and let 
$\St_r=\nabla(r\rho)$. Let $U$ be the subgroup of unipotent upper triangular matrices.
\begin{Proposition}\label{propnegligible}
 If $A$ has a good Grosshans filtration, then
\begin{itemize}
 \item $H^i(\SL_N,M\otimes_\k\k[\SL_N/U])$ vanishes for $i\gg0$,
\item $H^1(\SL_N,M\otimes_\k\St_r\otimes_\k\St_r\otimes_\k \k[\SL_N/U])$ vanishes for $r\gg0$.
\end{itemize}
\end{Proposition}

Define the `Grosshans filtration dimension' of a nonzero $M$ to be the minimum 
$d$ for which $H^{d+1}(\SL_N,M\otimes_\k \k[\SL_N/U])$
vanishes. As $(\St_r\otimes_\k\St_r)^G=\k$, we have a natural map $V\to V\otimes_\k\St_r\otimes_\k\St_r$
for any $G$-module $V$. In the theorem
one may start with $N_0:=M\otimes_\k\St_r\otimes_\k\St_r$. The cokernel of $M\to N_0$ will then have a lower 
Grosshans filtration dimension. And Grosshans filtration dimension zero implies good Grosshans filtration 
\cite[Proposition 28]{FvdK}.

\begin{Remark}
 In Proposition \ref{propnegligible} it would suffice to tensor once with $\St_r$. Our formulation is adapted to the proof
of Theorem \ref{maingood}.
\end{Remark}

As in \cite{Srinivas vdK} the method of proof of Theorem \ref{maingood}
is based on the functorial resolution \cite{twisted}
of the ideal of the diagonal in
$Z\times Z$  when $Z$ is a Grassmannian of subspaces of $\k^N$.
This is used inductively to study equivariant sheaves on a product $X$
of such Grassmannians.
That leads to a special case of the theorems, with $A$ equal to the 
Cox ring of $X$, multigraded by the Picard group $\Pic(X)$, and $M$ compatibly
 multigraded.
Next one treats cases when on the same $A$ the multigrading is replaced 
with a `collapsed' grading with smaller value group
and $M$ is only required to be 
 multigraded compatibly with this new grading. Here the trick is that 
an associated graded of $M$ has a multigrading that is collapsed a little less.
The suitably multigraded Cox rings are then used as in \cite{Srinivas vdK}
to cover the general case
\ref{maingood}. 

Recall that section \ref{consequences} gives some consequences for earlier work.

\section{Recollections and conventions}
Some unexplained notations, terminology, properties, \ldots can be
found in \cite{Jantzen}. Until section 8 the group $G$ is either $\GL_N$ or $\SL_N$.
Some things are best told with $\GL_N$, but the conclusion of Proposition \ref{propnegligible}
refers only to the $\SL_N$-module structure. 

\label{G=GL}
First let
 $G=\GL_N$, with $B^+$ its subgroup of upper triangular matrices,
$B^-$ the opposite Borel subgroup, $T=B^+\cap B^-$
the diagonal subgroup, $U=U^+$ the unipotent radical of $B^+$.
The roots of $U$ are positive, contrary to the  Aarhus convention followed in \cite{FvdK}.
The character group $X(T)$ has a basis $\epsilon_1$ \ldots, $\epsilon_N$
with $\epsilon_i(\diag(t_1,\ldots,t_N))=t_i$. 
An element $\lambda=\sum_i \lambda_i\epsilon_i$ of $X(T)$ is often denoted 
$(\lambda_1,\ldots,\lambda_N)$. It is called a polynomial weight if the
$\lambda_i$ are nonnegative. It is called a dominant weight if $\lambda_1\geq
\cdots\geq\lambda_N$. It is called anti-dominant if $\lambda_1\leq
\cdots\leq\lambda_N$. The fundamental weights $\varpi_1$, \dots , $\varpi_N$
are given by $\varpi_i=\sum_{j=1}^i\epsilon_j$.
If $\lambda\in X(T)$ is dominant, then $\ind_{B^-}^G(\lambda)$ is the
\emph{dual Weyl module} or \emph{costandard module} $\nabla_G(\lambda)$,
or simply $\nabla(\lambda)$, with highest weight $\lambda$. 
The \emph{Grosshans height} of $\lambda$ is
$\hgt(\lambda)=\sum_i(N-2i+1)\lambda_i$. It extends to a homomorphism
$\hgt:X(T)\otimes\Q\to\Q$.
The determinant representation has weight $\varpi_N$ and one has
$\hgt(\varpi_N)=0$.
Each positive root $\beta$ has $\hgt(\beta)>0$.
If $\lambda$ is a dominant polynomial weight, then $\nabla_G(\lambda)$
is called a \emph{Schur module}. If $\alpha$ is a partition with at most $N$ parts
then we may view it as a dominant polynomial weight and
the \emph{Schur functor} $S^\alpha$  maps $\nabla_G(\varpi_1)$ to 
$\nabla_G(\alpha)$. (This is the convention followed in \cite{twisted}.
In \cite{Akin} the same Schur functor is labeled with the conjugate
partition $\tilde \alpha$.)
The formula
$\nabla(\lambda)=\ind_{B^-}^G(\lambda)$ just means that $\nabla(\lambda)$
is obtained from the Borel-Weil construction:
$\nabla(\lambda)$ equals $H^0(G/B^-,{\mathcal L}_\lambda)$ for a certain
 line bundle ${\mathcal L}_\lambda$ on the
flag variety $G/B^-$.

Now consider the case $G=\SL_N$. \label{case SL} 
There are similar conventions for \hbox{$\SL_N$-modules}. For instance, the costandard
modules for $\SL_N$ are the restrictions of those for $\GL_N$.
The Grosshans height on $X(T)$ induces one on $X(T\cap\SL_N)\otimes\Q$. 
The \emph {multicone} $\k[\SL_N/U]$ consists of the $f$ in the coordinate ring
$\k[\SL_N]$ that satisfy $f(xu)=f(x)$ for $u\in U$. As an 
$\SL_N$-module it is the direct sum of all costandard modules. It is also
a finitely generated algebra \cite{Kempf Ramanathan}, \cite{Grosshans contr}, \cite[Lemma 23]{FvdK}.
Note that $\k[\SL_N/U]$ is $\SL_N$-equivariantly isomorphic to $\k[\SL_N/U^-]$, so that here it does not matter 
whether one follows the  Aarhus convention or not.

\begin{Definition}
A \emph {good filtration} of a $G$-module $V$ is a filtration
$0=V_{\leq -1}\subseteq V_{\leq 0} \subseteq V_{\leq 1}\ldots$ by
$G$-submodules $V_{\leq i}$ with $V=\cup_iV_{\leq i}$, so
that its associated graded $\gr V$ is a direct sum of costandard modules.

A Schur filtration of a polynomial $\GL_N$-module $V$ is a filtration
$0=V_{\leq -1}\subseteq V_{\leq 0} \subseteq V_{\leq 1}\ldots$ by
$\GL_N$-submodules with $V=\cup_iV_{\leq i}$, so
that its associated graded $\gr V$ is a direct sum of Schur modules.
The \emph{Grosshans filtration} of $V$ is the filtration with
$V_{\leq i }$ the largest $G$-submodule of $V$ whose weights
$\lambda$ all satisfy $\hgt(\lambda)\leq i$. Good filtrations and
Grosshans filtrations for $\SL_N$-modules are defined similarly. 
The literature contains more restrictive definitions 
of good/Schur filtrations. Ours are the right ones
when dealing with representations that need not be finitely generated over $\k$.

Let $M$ be a $G$-module provided with the Grosshans filtration.
Recall from \cite{FvdK} that  $M$ has \emph{good Grosshans filtration} if
the embedding of $ \gr M$ into $\hull_\nabla(\gr M)=\ind_{B^-}^GM^U$ is an isomorphism.
Then $ \gr M$ is a direct sum of modules of the form $\nabla(\lambda)\otimes_\k J(\lambda)$
with $G$ acting trivially on $J(\lambda)$. The $J(\lambda)$ need not be flat.
If they are all free then we are back at the case of a good filtration.
\end{Definition}

A $G$-module $M$ has good Grosshans filtration if and only if $H^1(\SL_N,M\otimes_\k\k[\SL_N/U])$
vanishes \cite[Proposition 28]{FvdK}. And $H^1(\SL_N,M\otimes_\k\k[\SL_N/U])$ vanishes if and only if $H^1(\SL_N,M\otimes_\k V)$ vanishes 
for every module $V$ with good filtration.
 A module with good filtration has good Grosshans filtration and is flat as a $\k$-module.
The tensor product of two modules with good filtration has good filtration \cite[Lemma B.9, II Proposition 4.21]{Jantzen}.
The tensor product of a module with good filtration 
and one with good Grosshans filtration thus
has good Grosshans filtration. If $M$ is a $G$-module, then $M\otimes \k[G]$ has a good Grosshans filtration
by \cite[I Lemma 4.7a]{Jantzen}. 
This may be used in dimension shift arguments.
 If $H^i(\SL_N,M\otimes_\k\k[\SL_N/U])$ vanishes, then so does $H^{i+1}(\SL_N,M\otimes_\k\k[\SL_N/U])$.
This follows from \cite[Proposition 28]{FvdK} and dimension shift. 
The following Lemma is also proved by dimension shift.

\begin{Lemma}\label{goodGgood}
 If $M$ has finite  Grosshans filtration dimension $d$ and $V$ has good filtration,
then $M\otimes V$ has   finite  Grosshans filtration dimension $\leq d$.\qed
\end{Lemma}

Note that $\k$ itself has good filtration.
So a module with good Grosshans filtration is $\SL_N$-acyclic, hence also $G$-acyclic when $G=\GL_N$.

\section{Gradings}
Let $\Theta=\Z^r$ with standard basis $e_1$, \ldots , $e_r$.
We partially order $\Theta$ by declaring that $I\geq J$ if $I_q\geq J_q$
for $1\leq q\leq r$. The \emph{diagonal} $\diag(\Theta)$ consists of the integer
multiples of the vector $E=(1,\ldots,1)$.
By a \emph{good $G$-algebra} we mean a 
finitely generated commutative $\k$-algebra $A$ on which $G$ acts rationally
by $\k$-algebra automorphisms so that $A$ has a good filtration as a $G$-module.
We say that $A$ is a \emph{good \GD-algebra} if moreover $A$ is $\Theta$-graded
by $G$-submodules,
$$A=\bigoplus_{I\in\Theta,~I\geq0}A_I$$
 with 
\begin{itemize}
\item
 $A_IA_J\subset A_{I+J}$,
\item $A$
is generated over $A_0$ by the $A_{e_q}$,
\item $G$ acts trivially on $A_0$.
\end{itemize}
Motivated by the Segre embedding we define 
$$\diag(A)=\bigoplus_{I\in\diag(\Theta)}A_I$$ and $\Proj(A):=\Proj(\diag(A))$.
By an \AG-module we will mean a noetherian $A$-module $M$
with compatible $G$-action.
If moreover $M$ is $\Theta$-graded by $G$-submodules $M_I$ so that 
$A_IM_J\subset M_{I+J}$, then we call $M$ an \emph{\AGD-module}.

\begin{Definition}
We call a $G$-module $M$ \emph{negligible} if
\begin{itemize}
\item $H^i(\SL_N,M\otimes_\k\k[\SL_N/U])$ vanishes for $i\gg0$ and  
\item $H^1(\SL_N,M\otimes_\k\St_r\otimes_\k\St_r\otimes_\k\k[\SL_N/U])$ vanishes for $r\gg0$.
\end{itemize}
 In other words,
$M$ must have finite 
Grosshans filtration dimension
and $M\otimes_\k\St_r\otimes_\k\St_r$ must have good Grosshans filtration for $r\gg0$.
We will be interested in $AG$-modules being negligible. In particular a {good \GD-algebra} $A$ is itself negligible. 
\end{Definition}

\begin{Lemma}\label{two/three}
Let $$0\to M'\to M\to M''\to 0$$ be exact, with
$M'$ be negligible. Then  $M$ is negligible iff $M''$ is negligible.
\end{Lemma}
\begin{proof}
Use that if $H^i(\SL_N,V\otimes_\k\k[\SL_N/U])$ vanishes, then so does  $H^{i+1}(\SL_N,
V\otimes_\k\k[\SL_N/U])$.
\end{proof}
\begin{Lemma}
Let 
$0\to M_0\to M_1\to \cdots\to M_q\to 0$ be a complex of \AG-modules
whose homology modules $\ker(M_i\to M_{i+1})/\im(M_{i-1}\to M_{i})$
are negligible for $i=0,\dots,q$. If $M_i$ is negligible for $i<q$
then  $M_{q}$ is negligible.\qed
\end{Lemma}

\section{Picard graded Cox rings}\label{section 4}
If $V$ is a finitely generated projective $\k$-module, we denote its dual by $V^\#$.
For $1\leq s\le N$, let $\Gr(s)$ be the Grassmannian scheme over $\k$ parametrizing rank
$s$ subspaces of the 
dual $\nabla(\varpi_1)^\#$ of the
defining representation of $\GL_N$. If one does base change to an algebraically closed field $\F$,
then one gets the Grassmannian variety $\Gr(s)_{\F}$ over $\F$ parametrizing 
$s$-dimensional subspaces of the 
dual $\nabla(\varpi_1)^\#$ of the
defining representation of $\GL_N$. We think of $\Gr(s)$ as a constant family parametrized by $\Spec(\k)$.
Note that we often suppress the base ring $\k$ in the notation.
The point is that we will argue in a manner which minimizes the need to pay attention
to the base ring. 
Let $\cO(1)$ denote the usual
ample sheaf on $\Gr(s)$, corresponding with
a generator of the Picard group of $\Gr(s)_{\F}$. 
We wish to view it as
a $G$-equivariant sheaf. To this end consider $G=\GL_N$ with its parabolic subgroup
$P=\{\;g\in G\mid g_{ij}=0\mbox{ for }i>N- s,~j\leq N- s\;\}$ and identify $\Gr(s)$
with $G/P$. Then a $G$-equivariant vector bundle is the associated bundle
of its fiber over $P/P$, where this fiber is a $P$-module. For the line bundle
$\cO(1)$ we let $P$ act by the weight $\varpi_N-\varpi_{N-s}$ on the fiber
over $P/P$. With this convention
$\Gamma(\Gr(s),\cO(1))$ is the Schur module $\nabla(\varpi_s)$, cf.\
\cite[II.2.16]{Jantzen}. 
More generally, for $n\geq0$ one has $\Gamma(\Gr(s),\cO(n))=\nabla(n\varpi_s)$.
So $$A\langle s\rangle =\oplus_{n\geq0}\Gamma(\Gr(s),\cO(n))$$
is a good $G\Z$-algebra. Recall that $\Theta=\Z^r$.
Let $1\leq s_i\le N$ be given for $1\leq i\leq r$. Repetitions are definitely allowed.
Then the Cox ring $A\langle s_1\rangle \otimes\cdots\otimes A\langle s_r\rangle $ of 
$\Gr(s_1)\times\cdots\Gr(s_r)$ is a good \GD-algebra.
We put $C=C_0\otimes A\langle s_1\rangle \otimes\cdots\otimes A\langle s_r\rangle $,
where $C_0$ is a polynomial algebra on finitely many generators over $\k$ with trivial
$G$-action. Then $C$ is also a good \GD-algebra. Here $G$ may be either $\SL_N$ or $\GL_N$.
We wish to prove

\begin{Proposition}\label{Picardgraded}
Every \CGD-module is negligible.
\end{Proposition}

The proof will be by induction on the rank $r$ of $\Theta$.
It will be finished in \ref{endproofPicardgraded}. Notice that the property of being negligible depends only on
the \hbox{$\SL_N$-module} structure. In particular, a shift in the grading makes no difference.
As base of the induction we use

\begin{Lemma}
A \CG-module $M$ that is noetherian over $C_0$ is negligible. 
\end{Lemma}
\begin{proof}We may view $M$ as $C_0G$-module and forget that $M$ is a $C$-module. 
Say $G=\SL_N$.
First let us show that $H^i(G,M\otimes_\k \k[G/U])$ vanishes for $i\gg0$.
We claim that it only depends on the weights of $M$ how large $i$ must be taken.
Say all weights of $M$ have length at most $R$. We argue by induction on the highest weight of $M$.
To perform the induction, we first choose a total order on weights of length at most $R$, 
that refines the usual dominance order of \cite[II.1.5]{Jantzen}. 
Initiate the induction with $M=0$. For the induction step,
consider the highest weight $\mu$ in $M$ and let $M_{\mu}$ be its weight space. We let $G$ act trivially on $M_{\mu}$. 
Now, by  \cite[Proposition 21]{FvdK}
$\Delta(\mu)_{\Z}\otimes_{\Z}M_{\mu}$ maps to $M$, and the kernel and the cokernel of this map have lower highest weight. 
So we still need to see that $\Delta(\mu)_{\Z}\otimes_{\Z}M_{\mu}$ itself has the required property.
But $\nabla(\mu)_{\Z}\otimes_{\Z}M_{\mu}$ has it by the universal coefficient theorem \cite[A.X.4.7]{Bourbaki}, 
and the natural map from $\Delta(\mu)_{\Z}\otimes_{\Z}M_{\mu}$ to 
$\nabla(\mu)_{\Z}\otimes_{\Z}M_{\mu}$ also
has  kernel and cokernel of  lower highest weight. 
All in all there is an effective bound for $i$ in terms of the weight structure of $M$.

Now we still have to show that $H^j(G,M\otimes_\k \nabla(r\rho)\otimes_\k \nabla(r\rho)\otimes_\k \k[\SL_N/U])$  vanishes for $j>0$
when $r$ is large enough. 
First let $V$ be a $C_0G$-module that is obtained by tensoring a flat noetherian $G_{\Z}$-module 
$V_{\Z}$ with a $C_0$-module 
on which $G$ acts trivially. 
Then to show that $V$ has the required property we wish to invoke the universal coefficient theorem \cite[A.X.4.7]{Bourbaki}
again. We take $r$ so large that $r\rho-\mu$ is dominant for all weights $\mu$ of $V$.
Look at the $H^j(G_\Z,V_\Z\otimes\nabla_\Z(r\rho)\otimes\nabla_\Z(r\rho)\otimes\nabla_\Z(\nu))$ for $j>0$,  $\nu$ dominant.
They are noetherian $\Z$-modules.
The corresponding groups vanish over a field $\F$. Indeed in view of \cite{Wang} the reasoning
in  \cite[\S 3]{CPSvdK} shows
 that $V_\F\otimes\nabla_\F(r\rho)\otimes\nabla_\F(r\rho)$ has good filtration because 
$r$ is so large that $r\rho-\mu$ is dominant for all weights $\mu$ of $V$.
But then the above $H^j(\dots)$ must vanish over $\Z$ for such $r$.
So $H^j(G,V\otimes\nabla(r\rho)\otimes\nabla(r\rho)\otimes\k[\SL_N/U])$  vanishes for $j>0$ by the universal coefficient theorem.
Now one may argue by induction on the highest weight of $M$ again. \end{proof} 

\begin{Remark}\label{universal coefficient reasoning}
 This kind of reasoning with the universal coefficient theorem is needed in many places to extend facts from fields to our 
base ring $\k$. We may use it tacitly.
\end{Remark}

\begin{Notation}
For $1\leq q\leq r$ we denote by $C^{\widehat  q}$ the subring 
$\bigoplus_{I_q=0}C_I$.
\end{Notation}

We further assume $r\geq1$. The inductive hypothesis then gives:
\begin{Lemma}\label{inductiveassumption}Let $1\leq q\leq r$.
If the \CGD-module $M$ is noetherian over the subring $C^{\widehat  q}$,
then $M$ is negligible.
\end{Lemma}

\section{Coherent sheaves}
We now have $\Proj(C)=\Spec(C_0)\times \Gr(s_1)\times\cdots\Gr(s_r)$.
Call the projections of $\Proj(C)$ onto its respective factors 
$\pi_0$, \ldots , $\pi_r$.
For $I\in\Theta$ define the coherent sheaf 
$\cO(I)=\bigotimes_{i=1}^r\pi_i^*(\cO(I_i))$.
So $C=\bigoplus_{I\geq0}\Gamma(\Proj(C),\cO(I))$.
For a \CGD-module $M$ let $M^\sim$ be the coherent
$G$-equivariant \cite[II F.5]{Jantzen} sheaf on $\Proj(C)$
constructed as in \cite[II 5.11]{Hartshorne} from the $\Z$-graded
module $\diag(M):=\bigoplus_{I\in\diag(\Theta)}M_I$. 

Conversely, to a coherent sheaf $\cM$ on $\Proj(C)$, we associate
the $\Theta$-graded $C$-module 
$$\Gamma_*(\cM)=\bigoplus_{I\geq0}\Gamma(\Proj(C),\cM(I)),$$
where $\cM(I)=M\otimes\cO(I)$.
We also put $H_*^t(\cM)=\bigoplus_{I\geq0}H^t(\Proj(C),\cM(I))$.

\begin{Notation}\label{exterior}
On a product like $X\times Y$ an exterior tensor product
$\pi_1^*(\cF)\otimes\pi_2^*(\cM)$ is denoted $\cF\boxtimes\cM$.
\end{Notation}

\begin{Lemma}[K\"unneth]\label{Kunneth}
 Let $X$ and $Y$ be flat projective schemes over an affine scheme $S=\Spec(R)$. Let $\cF$ be a quasicoherent sheaf
on $X$ and $\cG$ one on $Y$. Assume $\cF$, $\cG$ are flat over $S$ and that $\Gamma(X,\cF)$ is flat over $R$.
If $H^i(X,\cF)$
vanishes for $i>0$, 
then $H^i(X\otimes_S Y,\cF\boxtimes \cG)=\Gamma(X,\cF)\otimes_R  H^i(Y, \cG)$
for all $i$.
\end{Lemma}
\begin{proof}
Use \cite[Theorem 14]{Kempf}.
\end{proof}

\begin{Lemma}
If $\cM$ is a $G$-equivariant coherent sheaf  on $\Proj(C)$,
then the  $H_*^t(\cM)$ are \CGD-modules.
\end{Lemma}
\begin{proof}
So we have to show that $H_*^t(\cM)$ is noetherian as a
$C$-module. Observe that $\Proj(C)$ has a finite affine cover, so that $H_*^t(\cM)$
vanishes for $t$ large. So we argue by descending
induction on $t$. Assume the result for all larger values of $t$.
By Kempf vanishing 
$\bigoplus_{q\geq0}\bigoplus_{n\geq0}H^q(Gr(s),\cO(i+n))$
is a noetherian $\bigoplus_{n\geq0}\Gamma(Gr(s),\cO(n))$
module, for any $i\in\Z$. Similarly, by Lemma \ref{Kunneth} we see that
$\bigoplus_{q\geq0}H_*^q(\Proj(C),\cO(I))$ is a noetherian $C$-module
for any $I\in \Theta$, generated by the $H^q(\Proj(C),\cO(I+J))$ with $0\leq J_j\leq |I_j|$.
Now write $\cM$ as a quotient of some $\cO(iE)^{\oplus a}$ and use
the long exact sequence
$$\cdots\to H_*^t(\cO(iE)^{\oplus a})\to H_*^t(\cM)\to H_*^{t+1}
(\dots)\to\cdots$$
to finish the induction step.
\end{proof}

\begin{Notation}
If $M$ is a $\Theta$-graded module and $I\in\Theta$, then $M(I)$ is the
$\Theta$-graded module with $M(I)_J=M_{I+J}$.
Further $M_{\geq I}$ denotes $\bigoplus_{J\geq I}M_J$.
\end{Notation}

\begin{Lemma}\label{diagtest}
If $I\geq0$, then the ideal $C_{\geq I}$ of $C$ is generated by $C_I$.\\
If $M$ is a \CGD-module with $M_{nE}=0$ 
for $n\gg 0$, then  $M_{\geq nE}=0$ for $n\gg 0$.
\end{Lemma}

\begin{proof}
The ideal is generated by $C_I$ because $C$ is generated over $C_0$ by the
$C_{e_i}$.
Let $m\in M_I$. Choose $J\geq 0$ with $I+J\in\diag(\Theta)$. Then $mC_{J+qE}$
vanishes for $q\gg 0$, so $(mC)_{\geq I+J+qE}=0$ for $q\gg 0$.
Now use that $M$ is finitely generated over $C$.
\end{proof}

\begin{Lemma}
If $M$ is a \CGD-module, then there is an $n_0$
so that if $I=nE=(n,\dots,n)\in\Theta$ with $n>n_0$,
then $M_{\geq I}=\Gamma_*(M^\sim)_{\geq I}$.
\end{Lemma}

\begin{proof}
Recall \cite[II Exercise 5.9]{Hartshorne}, \cite[II 2.5.4, 2.7.3, 2.7.11]{ega} 
that we have a natural map
$\diag(M)\to \diag(\Gamma_*(M^\sim))$
whose kernel and cokernel live in finitely many degrees.
Consider the maps $f:\diag(M)\otimes_{\diag(C)}C\to M$
and $g:\diag(M)\otimes_{\diag(C)}C\to \Gamma_*(M^\sim)$.
If $N$ is the kernel or cokernel of $f$ or $g$ then $N_{nE}=0$ for 
$n\gg 0$. Now apply the previous lemma.
\end{proof}

\begin{Lemma}
If $M$ is a \CGD-module and $I\in\Theta$, then $M/M_{\geq I}$ is
negligible.
\end{Lemma}
\begin{proof}
As $M$ is finitely generated over $C$, there is $J<I$ with
$M=M_{\geq J}$. Now note that for $1\leq q\leq r$ and $K\in \Theta$
the module $M_{\geq K}/M_{\geq K+e_q}$ is negligible by \ref{inductiveassumption}.
\end{proof}

\begin{Lemma}\label{half}
If $M$ is a \CGD-module and $I\in\Theta$, then $M$ is negligible iff $M_{\geq I}$ is
negligible.
\end{Lemma}
\begin{proof}
As a $\k G$-module $M$ is the direct sum of $M/M_{\geq I}$ and $M_{\geq I}$.
\end{proof}

\begin{Definition}
In view of the above we call an equivariant coherent sheaf $\cM$
on $\Proj(C)$
negligible when $\Gamma_*(\cM)$ is negligible.
\end{Definition}

The following Lemma is now clear:

\begin{Lemma}Let $I\in\Theta$.
A $G$-equivariant coherent sheaf $\cM$ on $\Proj(C)$ is negligible if and only
if $\cM(I)$ is negligible.
\end{Lemma}

\begin{Lemma}
Let $$0\to\cM'\to\cM\to\cM''\to0$$
be an exact sequence of $G$-equivariant coherent sheaves on $\Proj(C)$.
There is $I\in\Theta$ with 
$$0\to\Gamma_*(\cM')_{\geq I}\to\Gamma_*(\cM)_{\geq I}
\to\Gamma_*(\cM'')_{\geq I}\to0$$
exact.
\end{Lemma}
\begin{proof}The line bundle $\cO(E)$ is ample. 
Apply Lemma \ref{diagtest} to the homology sheaves of the
complex $$0\to\Gamma_*(\cM')\to\Gamma_*(\cM)
\to\Gamma_*(\cM'')\to0.$$
\end{proof}

\begin{Lemma}
For every $I\in\Theta$ the sheaf $\cO(I)$ is negligible.

\end{Lemma}
\begin{proof}
Use that $C$ is negligible.
\end{proof}

\section{Resolution of the diagonal}
We write $X=\Proj(C)$, $Y=\Proj(C^{\hat r})$, $Z=\Gr(s)$, where $s=s_r$.
So $X=Y\times Z$.
We now recall the salient facts from \cite{twisted}, \cite{diagonal}
 about the 
functorial resolution of the diagonal in $Z\times Z$.
The fact that our base ring is now $\k$ is not a problem. In \cite{twisted}
one already works over a noetherian base, and  \cite{diagonal} is just extra.
But let us temporarily take $Z$ to be the Grassmannian $\Gr(s)_\Z$ over $\Z$.
Later we will do the base change from $\Z$ to $\k$.
As $Z$ is the Grassmannian that parametrizes the $s$-dimensional subspaces
of 
$\nabla(\varpi_1)^\#$, we have the tautological exact sequence of $G$-equivariant
vector bundles
on $Z$:
$$0\to\cS\to \nabla(\varpi_1)^\#\otimes \cO_Z\to \cQ\to 0,$$
where $\cS$ has as fiber above a point the subspace $V$
that the point parametrizes, and $\cQ$
has as fiber above this same point the quotient 
$\nabla(\varpi_1)^\#/V$.
Let $\pi_1$, $\pi_2$ be the respective projections $Z\times Z\to Z$.
Then the composite of the natural maps 
$\pi_1^*(\cS)\to \nabla(\varpi_1)^\#\otimes \cO_{Z\times Z}$
and 
$\nabla(\varpi_1)^\#\otimes \cO_{Z\times Z}\to\pi_2^*(\cQ)$ defines a section
of the vector bundle $\cHom(\pi_1^*(\cS),\pi_2^*(\cQ))$ whose zero scheme 
is the diagonal $\diag(Z)$ in $Z\times Z$. Dually, we get an exact sequence 
$\cHom(\pi_2^*(\cQ),\pi_1^*(\cS))\to \cO_{Z\times Z}\to \cO_{\diag Z}\to 0$,
where $\cO_{\diag Z}$ is the quotient by the ideal sheaf defining the diagonal.
As the rank $d$ of the vector bundle $\cE=\cHom(\pi_2^*(\cQ),\pi_1^*(\cS))$ equals
the codimension of $\diag(Z)$ in $Z\times Z$, the Koszul complex
$$0\to \bigwedge^d\cE\to\cdots\to\cE\to\cO_{Z\times Z}\to \cO_{\diag Z}\to 0$$
is exact. Now each $\bigwedge^i\cE$ has a finite filtration whose associated
graded is 
$$\bigoplus S^\alpha \pi_1^*(\cS)\otimes (S^{\tilde\alpha}\pi_2^*(\cQ))^\#,$$
where $\alpha $ runs over partitions of $i$ with at most $\rank(\cS)$ parts,
so that moreover the conjugate partition $\tilde \alpha$ has at most $\rank(\cQ)$
parts. Now do the base change from $\Z$ to $\k$, so that $Z$ is defined over $\k$.
The Koszul complex remains exact as it was flat over $\Z$. The expression for the associated graded 
of $\bigwedge^i\cE$ also remains valid.

\subsection*{Plan}Now the plan is this: Let $\pi_{1,2}$ be the projection
of  $Y\times Z\times Z$  onto the product $Y\times Z$ of the first two factors,
let $\pi_2$ be the projection onto the middle factor $Z$, and so on.
If $M$ is a \CGD-module, tensor the pull-back along $\pi_{2,3}$ of the
 Koszul complex with
$\pi_{1,3}^*(M^\sim)$, take a high Serre twist
and then the direct image along $\pi_{1,2}$ to $X$. 
On the one hand $(\pi_{1,2})_*(\pi_{1,3}^*(M^\sim)\otimes \cO_{\diag Z})$
is just $M^\sim$, but on the other hand the salient facts above allow us to express it in
terms of negligible  \CGD-modules. This will prove that $M$ is negligible.
We now proceed with the details.

\begin{Remark}
Instead of functorially resolving the diagonal in $Z\times Z$, we could have
functorially resolved the diagonal in $X\times X$.
\end{Remark}

Recall from \ref{exterior} that $\cF\boxtimes\cM$ denotes an exterior tensor product of sheaves.

\begin{Lemma}
Let $\cF$ be a $G$-equivariant coherent sheaf on $Y$, and $\alpha$ a partition of
$i$ with at most $s$ parts, $i\geq0$.
The sheaf $\cF\boxtimes S^\alpha(\cS)$ on $X=Y\times Z$ is
negligible.
\end{Lemma}
\begin{proof}
By the inductive assumption 
$$\Gamma_*(\cF)=\bigoplus_{I\in\Z^{r-1},~I\geq0}
\Gamma(Y,\cF(I))$$
is a
$C^{\widehat r}$-module
with finite Grosshans filtration dimension and 
$\St_r\otimes\St_r\otimes\Gamma_*(\cF)$ has good Grosshans filtration for $r\gg0$.
The vector bundle  $\cS$ on $Z=G/P$ is associated with the
irreducible $P$-representation with lowest weight $-\epsilon_{N-s+1}$.
This representation may be viewed as $\ind_{B^+}^P(-\epsilon_{N-s+1})$,
where $-\epsilon_{N-s+1}$ also stands for
the one dimensional $B^+$ representation with weight $-\epsilon_{N-s+1}$.
Say $\rho:P\to P^-$ is the isomorphism which sends
a matrix to its transpose inverse. Then 
$\ind_{B^+}^P(-\epsilon_{N-s+1})=
\rho^*\ind_{B^-}^{P^-}(\epsilon_{N-s+1})$.
One finds that $S^\alpha(\cS)$ is associated with
$\rho^*\ind_{B^-}^{P^-}\left(\sum_i\alpha_i\epsilon_{N-s+i}\right)=
\ind_{B^+}^P\left(-\sum_i\alpha_i\epsilon_{N-s+i}\right)$.
(This is the rule $S^\alpha(\nabla_{\GL_s}(\varpi_1))=\nabla_{\GL_s}(\alpha)$
in disguise.)
Then $S^\alpha(\cS)(n)$ is associated with 
$\ind_{B^+}^P\left(-\sum_i\alpha_i\epsilon_{N-s+i}+n\varpi_N-n\varpi_{N-s}\right)$.
For $n\geq \alpha_1$ the weight 
$-\sum_i\alpha_i\epsilon_{N-s+i}+n\varpi_N-n\varpi_{N-s}$ is an anti-dominant
polynomial weight, so 
$\sum_{n\geq\alpha_1}\Gamma(Z,S^\alpha(\cS)(n))$
has a good filtration by transitivity of induction \cite[I.3.5, I.5.12]{Jantzen}.
Moreover,  we have $H^i(Z, S^\alpha(\cS)(n))=0$ for $n\geq \alpha_1$, $i>0$,
by the universal coefficient theorem \cite[A.X.4.7]{Bourbaki}, 
Kempf vanishing \cite[Proposition II.4.6 (c)]{Jantzen}
and the spectral sequence for 
$\ind_P^G\circ\ind_{B^+}^P$ \cite[Proposition I.4.5(c)]{Jantzen}.
Fix $I=(0,\dots,0,\alpha_1)$ and consider $\Gamma_*(\cF\boxtimes S^\alpha(\cS))_{\geq I}$.
By \cite[Theorem 12]{Kempf} we may rewrite it as $\sum_{n\geq\alpha_1}\Gamma(Z,\Gamma_*(\cF)\otimes_\k
S^\alpha(\cS)(n))$, which equals $\Gamma_*(\cF)\otimes_\k\sum_{n\geq\alpha_1}\Gamma(Z,
S^\alpha(\cS)(n))$ by the universal coefficient theorem.
So by Lemma \ref{goodGgood} we get that $\Gamma_*(\cF\boxtimes S^\alpha(\cS))_{\geq I}$ 
has finite Grosshans filtration dimension 
and similarly $\St_r\otimes\St_r\otimes\Gamma_*(\cF\boxtimes S^\alpha(\cS))_{\geq I}$
has good Grosshans filtration for $r\gg0$. Apply Lemma \ref{half}.
\end{proof}

\begin{Lemma}
For $n\gg 0$ the sheaf 
$$(\pi_{12})_*\left(\pi_{13}^*(M^\sim)\otimes\bigg(\cO(nE)\boxtimes\cO(n)\bigg)
\otimes\pi_{23}^*(\bigwedge^i\cE)\right)$$ is
negligible.
\end{Lemma}
\begin{proof}The sheaf $\cO(E)\boxtimes\cO(1) $ is ample.
So \cite[Theorem 8.8]{Hartshorne} the sheaf in the Lemma
has a filtration with layers of the form
$$(\pi_{12})_*\left(\pi_{13}^*(M^\sim)\otimes\bigg(\cO(nE)\boxtimes\cO(n)\bigg)
\otimes\pi_{23}^*\bigg( S^\alpha (\cS)\boxtimes \cG\bigg)\right).$$
Say $f:Y\times Z\to Y$ is the projection. It is flat and 
by \cite[Proposition 9.3]{Hartshorne} we have $(\pi_{12})_*\circ\pi_{13}^*=f^*\circ f_*$. Now use this and
a projection formula for $(\pi_{12})_*$ to rewrite the layer
in the form $(\cF\boxtimes S^\alpha(\cS))(I)$
for some $I\in\Theta$, with $I$ depending on $n$.
\end{proof}

\begin{Lemma}
The Koszul complex
$$0\to \bigwedge^d\cE\to\cdots\to\cE\to\cO_{Z\times Z}\to \cO_{\diag Z}\to 0$$
remains exact when applying $\pi_{13}^*(M^\sim)\otimes\pi_{23}^*(?)$.
\end{Lemma}
\begin{proof}One is basically saying that
$\pi_{13}^*(M^{\sim})$ and
$\cO_{\diag Z}$
are Tor independent quasi-coherent sheaves on $Z\times Z$.
This is local and can be checked by computing in suitable coordinates.
We argue more globally.
Let $j$ be the isomorphism $Y\times\diag Z\to Y\times Z$ induced by $\pi_{13}$.
Write $K_\bullet=0\to K_d\to\cdots\to K_0\to 0$ for 
$0\to \bigwedge^d\cE\to\cdots\to\cE\to\cO_{Z\times Z}\to  0$.
Let $F_\bullet=\cdots\to F_1\to F_0\to 0$ be a resolution of $M^\sim$ by vector bundles.
Consider the homology of the total complex of the double complex 
$\pi_{13}^*(F_\bullet)\otimes \pi_{23}^*K_\bullet$. On the one hand this homology is the
homology of $\pi_{13}^*(F_\bullet)\otimes\pi_{23}^*\cO_{\diag Z}$, which is just
$j^*(F_\bullet)$. So it is concentrated in degree zero, where its homology is $j^*(M^\sim)$.
On the other hand it is the homology of $\pi_{13}^*(M^\sim)\otimes\pi_{23}^*(K_\bullet)$.
\end{proof}

\subsection*{End of proof of Proposition \ref{Picardgraded}}
Proposition \ref{Picardgraded} now follows from

\begin{Lemma}\label{endproofPicardgraded}
$M^\sim$ is negligible.
\end{Lemma}
\begin{proof}
From the Koszul complex and the two previous Lemmas we conclude \cite[Theorem 8.8]{Hartshorne}
that for $n\gg 0$ 
the sheaf $$(\pi_{12})_*\left(\pi_{13}^*(M^\sim)\otimes\bigg(\cO(nE)\boxtimes\cO(n)\bigg)
\otimes\pi_{23}^*(\cO_{\diag(Z)})\right)$$ is
negligible. This sheaf equals $M^\sim(I)$ for some $I\in \Theta$.
\end{proof}

\section{Differently graded Cox rings}
Let $c:\{1,\dots,r\}\to\{1,\dots,q\}$ be surjective.
Put $\Lambda=\Z^q$. We have a \emph{contraction map}, also denoted $c$,
from $\Theta$ to $\Lambda$ with $c(I)_j=\sum_{i\in c^{-1}(j)}I_i$.
Through this contraction we can view our $\Theta$-graded $C$ as $\Lambda$-graded.
We now have the following generalization of Proposition \ref{Picardgraded}:

\begin{Proposition}\label{retractedgraded}
Every \CGL-module is negligible.
\end{Proposition}
\begin{proof}
This will be proved by descending induction on $q$, with fixed $r$.
The case $q=r$ is clear. So let  $q<r$ and assume
the result for larger values of $q$.
We may assume $c(r-1)=c(r)=q$. (Otherwise rearrange the factors.)
Recall $X=\Proj(C)$, $X=Y\times Z$, with $Y=\Proj(C^{\widehat r})$, 
$Z=\Proj(A\langle s\rangle)$.

\begin{Notation}
Let $\m$ be the irrelevant 
 ideal $\bigoplus_{i>0}A\langle s\rangle _i$
of $A\langle s\rangle$. If $M$ is a \CGL-module, put $M_{\geq i}=\m^iM$, and
$\gr^iM=M_{\geq i}/M_{\geq i+1}$. If $I\in \Lambda$, put $(M_I)_{\geq i}=M_I\cap \m^iM$,
and $\gr^iM_I=(M_I)_{\geq i}/(M_I)_{\geq i+1}$.
We put a $\Z^{q+1}$-grading on $\gr M=\bigoplus_i\gr^i M$ with 
$$(\gr M)_I=\gr^{I_{q+1}}M_{(I_1,\dots,I_{q-1},I_q+I_{q+1})}.$$
In particular all this applies when $M=C$. Then $\gr C$ may be identified with $C$
and the $\Z^{q+1}$-grading on $\gr C$ is a contracted grading to which the inductive
assumption applies. Write $\Phi=\Z^{q+1}$. Then $\gr M$ is a $CG\Phi$-module.
\end{Notation}

Let $M$ be a \CGL-module.
By the inductive assumption $\gr M$ 
is negligible. So $\gr^iM_I$ is negligible. As the filtration on $M_I$ is finite, it follows that 
$M_I$ is negligible. But we need to do a little better.
We must show that when $i$ and $r$ are so big that $H^i(\SL_N,\gr M\otimes_\k\k[\SL_N/U])$  and  
$H^1(\SL_N,\gr M\otimes_\k\St_r\otimes_\k\St_r\otimes\k[\SL_N/U])$
vanish, then the same $i$  and $r$ work for all $M_I$ simultaneously. This is clear too, so $M$ is negligible.

This finishes the proof of Proposition \ref{retractedgraded}.
\end{proof}

\section{Variations on the Grosshans grading}\label{G=SL}
In this section we will be concerned with representations of $\SL_N$.
{\em Mutatis mutandis} everything also
applies to other connected reductive groups. We now write $G=\SL_N$,
with subgroups $B^+$, $B^-$, $T$, $U$ defined in the usual manner. (So they are now
the intersections with $\SL_N$ of the subgroups of $\GL_N$ that had these names.)
As explained above, the Grosshans graded $\gr V$ of an $\SL_N$-module $V$ has
a $\Z$-grading. We also need a $\Lambda$-graded version, where $\Lambda$ is 
the weight lattice of $\SL_N$.

Following Mathieu \cite{Mathieu G}
we choose a second linear height function $E:\Lambda\otimes \RR\to \RR$
with $E(\alpha)>0$ for every positive root $\alpha$, but now with $E$ injective
on $\Lambda$. We
define a total order on weights by first ordering them by Grosshans height,
then for fixed Grosshans height by $E$. With this total order, denoted $\leq$,  we put:

\begin{Definition}
If $V$ is a $G$-module, and $\lambda$ is  a weight,
then $V_{\leq\lambda}$ denotes the largest $G$-submodule all whose weights
$\mu$ satisfy $\mu\leq\lambda$ in the total order.
For instance, $V_{\leq0}$ is the module of invariants
$V^G$. Similarly $V_{<\lambda}$ denotes the largest $G$-submodule all 
whose weights
$\mu$ satisfy $\mu<\lambda$.  Note that $V\mapsto V_{\leq \lambda}$ is a truncation functor for a 
saturated set
of dominant weights \cite[Appendix A]{Jantzen}.
 So this functor fits in the usual
highest weight category picture.
We form the
$\Lambda$-graded module
$$\gr_{\Lambda} V=\bigoplus_{\lambda\in \Lambda}V_{\leq\lambda}/V_{<\lambda}.$$
Let us define the socle of a $B^+$-module $M$ to be $M^U$, because that is what it is when $\k$ is a field.
Each $\gr_\lambda V=V_{\leq\lambda}/V_{<\lambda}$
 has a $B^+$-socle
$(\gr_\lambda V)^U=V^U_\lambda$ of weight
$\lambda$. We always view $V^U$ as a $B^-$-module through restriction
(inflation)
along the
homomorphism $B^-\to T$.
Then $\gr_\lambda V$ embeds naturally in its `hull' $\hull_\nabla(\gr_\lambda V)
=\ind_{B^-}^GV^U_\lambda$.
This  hull has the same $B^+$-socle.
\end{Definition}

If $\lambda$ is not dominant, then $\gr_\lambda V$ vanishes, because
its socle vanishes.
Note that $\bigoplus_{\hgt(\lambda)=i}\gr_{\lambda} V$ is the associated graded of
a filtration of $\gr_i V$, where $\gr_{\lambda} V$ refers to a graded component of $\gr_{\Lambda} V$
and $\gr_i V$ to one of $\gr V$. Both $\gr_{\Lambda} V$ and $\gr V$ embed into the
 hull $\ind_{B^-}^GV^U$, which is $\Lambda$-graded. But while 
$\gr_{\Lambda} V$ is a $\Lambda$-graded submodule of the hull, 
$\gr V$ need only be a $\Z$-graded submodule. Both $\gr_{\Lambda} V$ and $\gr V$
contain the socle of the hull.

Although $\gr_\Lambda V$ need not coincide with $\gr V$ it shares some properties:

\begin{Lemma}\label{Lambda Grosshans}
\begin{enumerate}
\item If $A$ is a finitely generated $\k$-algebra, so is $\gr_\Lambda A$.
\item If $A$ has good Grosshans filtration, then $\gr_\Lambda A$ is isomorphic to $\gr A$ as a
$\k$-algebra.
\end{enumerate}
\end{Lemma}

\begin{proof}Let $A$ be a finitely generated $\k$-algebra.
By \cite[Lemma 25]{FvdK} the subalgebra $A^U$ is a finitely generated $\k$-algebra.
But $A^U$ is isomorphic to $(\gr_\Lambda A)^U$, so by \cite[Lemma 46]{FvdK}
the algebra $\gr_\Lambda A$ is finitely generated.
When $A$ has good Grosshans filtration, $\gr_i A$ is already a direct sum of 
modules of the form $\ind_{B^-}^G V_\lambda$, where $B^-$ acts on $V_\lambda$ with weight $\lambda$. 
So then passing to the associated graded of the filtration of $\gr_i A$ makes 
no difference. And the algebra structure on both $\gr A$ and $\gr_\Lambda A$
agrees with the algebra structure on the hull by the next Lemma.
\end{proof}

\begin{Lemma}\label{extend}
Let $A$ have a good Grosshans filtration, so that $\gr A=\hull_\nabla(\gr A)$.
Let $R$ be a $\Z$-graded algebra with $G$-action.
Assume that for each $i$ one has $R_i=(R_i)_{\leq i}$ in the Grosshans filtration. Then every $T$-equivariant
graded algebra homomorphism $R^U\to (\gr A)^U$ extends uniquely to a
$G$-equivariant graded algebra homomorphism $R\to \gr A$.
\end{Lemma}

\begin{proof}
Use that $\hull_\nabla(\gr_\Lambda A)$  is an induced module, so that we may use Frobenius reciprocity 
\cite[I Proposition 3.4]{Jantzen}.
\end{proof}

\section{Proof of the main result}
Let us now turn to the proof of Theorem \ref{maingood} for $\SL_N$.
Return to the notations introduced in section \ref{G=GL}. 
Thus $G=\GL_N$, with $T$ its maximal torus.
We assume the $\SL_N$-algebra $A$ has a good Grosshans filtration and $M$ is a
noetherian $A$-module on which $\SL_N$ acts compatibly.
Put $\Lambda=\Z^{N-1}$ and identify $\Lambda$ with a sublattice of $X(T)$
by sending $\lambda\in\Lambda$ to $\sum_i\lambda_i\varpi_i$.
Also identify $\Lambda$ with $X(T\cap\SL_N)$ through the restriction
$X(T)\to  X(T\cap\SL_N)$. Thus a dominant $\lambda\in\Lambda$ gets identified with
a polynomial dominant weight.
For such $\lambda$
we may embed $\gr_\lambda A$ or
$\gr_\lambda M$ into its hull which is the tensor product of 
the Schur module $\nabla_G(\lambda)$ with a
$\k$-module with trivial $G$ action.
On the Schur module $\nabla_G(\lambda)$ the center of $G$ acts through
$\lambda$. This makes it natural to
use the $\Lambda$-grading on $\gr_\Lambda A$ and $\gr_\Lambda M$ 
to extend the action from $\SL_N$ 
to $\GL_N$, making the center of $\GL_N$ act through $\lambda$
on the graded pieces $\gr_\lambda A$ and $\gr_\lambda M$. We do that.

As the algebra
$(\gr_\Lambda A)^U=(\gr A)^U$ is finitely generated by \cite[Lemma 25]{FvdK}
it is also generated by finitely many
weight vectors. Consider one such weight vector $v$, say of weight $\lambda$.
Clearly $\lambda$ is dominant.
If $\lambda=0$, map a polynomial ring $P_v:=\k[x]$ with trivial $G$-action to
 $\gr A$ by substituting $v$ for $x$. Also put $D_v:=1$.
Next assume $\lambda\neq0$. Let $\ell=N-1$ be the rank of $\Lambda$.
Recall the Cox rings $A\langle i\rangle$ of section \ref{section 4}. Define a $T$-action
on the $\Lambda$-graded algebra
$$P=\bigotimes_{i=1}^{\ell}A\langle i\rangle$$
by letting $T$ act on $\bigotimes_{i=1}^{\ell}\Gamma(\Gr(i),\cO(m_i))$
through weight $\sum_im_i\varpi_i$. So now we have a $G\times T$-action on $P$,
and the $T$-action corresponds with the $\Lambda$-grading.
Observe
that by  the tensor product
property \cite[Ch.~G]{Jantzen}
the algebra $P$ has a good filtration for the $G$-action.
Let $D$ be the scheme
theoretic kernel of $ \lambda$.
So $D$ has character group
$X(D)=X(T)/\Z \lambda$ and $D=\Diag(X(T)/\Z \lambda)$ in the notations of
\cite[I.2.5]{Jantzen}.
The subalgebra $P^{1\times D}$ is a graded algebra with
good filtration such that its subalgebra
$P^{U\times D}$ contains a polynomial algebra on one generator $x$ of weight
$\lambda\times \lambda$. In fact, this polynomial subalgebra contains all the
weight vectors in $P^{U\times D}$ whose weight is of the form $\nu\times \nu$.
 The other weight vectors in $P^{U\times D}$
have weights of the form $\mu\times\nu$
with $\nu$ an integer multiple of $\lambda$ and $\mu<\nu$. These other weight vectors
span an ideal in $P^{U\times D}$.
By lemma \ref{extend} one easily constructs
 a $G$-equivariant algebra homomorphism
$P^{1\times D}\to \gr_\Lambda A$ that maps $x$ to $v$.
Write it as $P_v^{1\times D_v}\to \gr_\Lambda A$, to stress the dependence on $v$.

The direct product $D$ of the $D_v$ is a diagonalizable group. It acts on
the tensor product $C$ of the finitely many $P_v$. This $C$ is $\Lambda$-graded.
We have a  graded algebra map $C^D\to \gr_\Lambda A$.
Observe that $\gr_\lambda A=\nabla(\lambda)\otimes J(\lambda)$ where  $J(\lambda)$ is
the $\lambda$ weight space of $A^U$, but with trivial $G$-action.
The map $\gr_\lambda C^D\to \gr_\lambda A$ is of the form 
$\nabla(\lambda)\otimes  J'(\lambda) \to\nabla(\lambda)\otimes  J(\lambda)$
with $G$ acting trivially on $J'(\lambda)$ also.
As each $J'(\lambda)\to J(\lambda)$ is surjective, so is $C^D\to \gr_\Lambda A$.
We have proved

\begin{Lemma}\label{cover}There is a graded $G$-equivariant surjection $C^D\to \gr_\Lambda A$,
where the
$G\times D$-algebra $C$  is a good $G\Lambda$ algebra as in \ref{retractedgraded}. 
\end{Lemma}

Now recall $M$ is a noetherian $A$-module on which $G$ acts compatibly, meaning
that the structure map $A\otimes M\to M$ is a map of $G$-modules.
Form the `semi-direct product ring' $A\ltimes M$ whose underlying $G$-module
is $A\oplus M$, with product given by
$(a_1,m_1)(a_2,m_2)=(a_1a_2,a_1m_2+a_2m_1)$. By Lemma \ref{Lambda Grosshans} $\gr_\Lambda(A\ltimes M)$ is
a finitely generated algebra, so we get

\begin{Lemma}
$\gr_\Lambda M$ is a noetherian $\gr_\Lambda A$-module.
\end{Lemma}

This is of course very reminiscent of the proof of the lemma
\cite[Theorem 16.9]{Grosshans book} telling that
$M^G$ is a noetherian module over the finitely generated $k$-algebra
$A^G$. We  will tacitly use
its counterpart for diagonalizable actions, cf. \cite{Borsari-Santos},
\cite[I.2.11]{Jantzen}.

Now this lemma implies that $C\otimes_{C^D}\gr_\Lambda M$ is a \CGL-module, so by Proposition 
\ref{retractedgraded} we get

\begin{Lemma}
$C\otimes_{C^D}\gr_\Lambda M$ is negligible
\end{Lemma}

Next we get 
 
\begin{Lemma}
The module $\gr_\Lambda M$ is negligible.
\end{Lemma}
\begin{proof}
Extend the $D$-action on $C$ to $C\otimes_{C^D}\gr_\Lambda M$ by using the trivial
action on the second factor. Then we have a $G\times D$-module structure
on $C\otimes_{C^D}\gr_\Lambda M$. As $D$ is diagonalizable,
$C^D$ is a direct summand of $C$ as a $C^D$-module \cite[I.2.11]{Jantzen}
and $(C\otimes_{C^D}\gr_\Lambda
M)^{1\times D}=\gr_\Lambda M$ is a direct summand of the $G$-module
$C\otimes_{C^D}\gr_\Lambda M$.
It follows that $\gr_\Lambda M$ is negligible.
\end{proof}

\subsection*{Proof of Proposition \ref{propnegligible}}
Fix $i$ and $r$  so big that $H^i(\SL_N,\gr M\otimes_\k \k[\SL_N/U])$ vanishes  and  
$H^1(\SL_N,\gr M\otimes_\k\St_r\otimes_\k\St_r\otimes_\k \k[\SL_N/U])$
vanishes. Enumerate the dominant weigths in
$\Lambda$ as
$\lambda_0$,  $\lambda_1$, \dots\ according to our total order on weights.
Note there are only finitely many dominant weights of given Grosshans height in $\Lambda$,
so that the order type of the set of dominant weights in $\Lambda$
is indeed just that of $\mathbb N$. (This would be false for the set of
dominant weights in $X(T)$.)
By induction on $\lambda$ we get vanishing of 
 $H^i(\SL_N, M_{\leq\lambda}\otimes_\k \k[\SL_N/U])$ and  
$H^1(\SL_N, M_{\leq\lambda}\otimes_\k\St_r\otimes_\k\St_r\otimes_\k \k[\SL_N/U])$ with the same the same $i$  and $r$.
As $G$-cohomology commutes with direct limits, $M$ is negligible.
\qed

\subsection*{Proof of Theorem \ref{maingood}}
A $\GL_N$-module has good Grosshans filtration if and only its restriction to $\SL_N$ has one.
One may embed $M$ into $M \otimes_\k\St_r\otimes_\k\St_r$
to start the resolution in Theorem \ref{maingood}. 
As the cokernel has a lower Grosshans filtration dimension, the Theorem follows.
\qed

\section{Consequences for earlier work}\label{consequences}
First let $\k$ be a noetherian ring containing a field $\F$ and let $G_\F$ be a geometrically reductive affine algebraic group
scheme over $\F$.
Write $G$ for the group scheme over $\k$ obtained by base change along $\F\to\k$.
Let $A$ be a finitely generated commutative $\k$-algebra on which $G$ acts
 rationally by $k$-algebra automorphisms. 
\begin{Theorem}[CFG when the base ring contains a field]\label{CFG:theorem}
\ \\
 $H^*(G,A)$ is a finitely generated $\k$-algebra.
\end{Theorem}

This is clear from \cite{TvdK} when $A$ is obtained by base change from an $\F$-algebra
with rational $G_\F$-action. 
Anyway, let us adapt the proof of \cite[Theorem 1.2]{TvdK}. First we will reduce to the case $G=\GL_N$.
Embed $G_\F$ in some $\GL_N$ over $\F$ and observe that the quotient $\GL_N/G_\F$ remains affine under base change
to $\k$, cf. \cite[I.5.5(1), I.5.4(5)]{Jantzen}. For group schemes over
 $\k$ geometric reductivity is no longer the right notion and 
we use power-reductivity \cite{FvdK} instead.

\begin{Lemma}
 Let $G$ be a  power-reductive flat affine algebraic group scheme over a ring $R$. 
For any commutative $R$-algebra $S$ the group scheme $G_S$ over $S$ is power-reductive.
\end{Lemma}
\begin{proof}
If $M$ is a module for $G_S$, then it is also a module for $G=G_R$ and with the same invariants \cite[Remark 52]{FvdK}.
By \cite[Proposition 10]{FvdK} $G_R$ has property (Int), which implies that $G_S$ has property (Int), hence is power
reductive. 
\end{proof}

\begin{Remark}
 The first line of the  proof of \cite[Proposition 10]{FvdK} ignores that \cite[Proposition 6]{FvdK}
refers to a finitely generated algebra. Indeed this finite generation hypothesis may safely be dropped,
as there is no finiteness hypothesis on the module in the definition of power reductivity. 
But we do not know that any algebra with $G$ action is a union of finitely generated invariant subalgebras.
That is because we simply do not know  if representations are always
locally finite. See \cite[Expos\'e VI, Remarque 11.10.1 in the 2011 edition]{SGA3}  for things that can go wrong
when $\k$ is not noetherian and $\k[G]$ is not a projective $\k$-module.
\end{Remark}

\subsection*{Proof of Theorem \ref{CFG:theorem}}
So we may argue as in \cite[Lemma 3.7]{reductive} that $H^*(G,A)=H^*(\GL_N,\ind_G^{\GL_N}(A))$, with
$\ind_G^{\GL_N}(A)=(A\otimes_\k \k[\GL_N])^G$ a finitely generated $\k$-algebra. 
This shows we may further assume $G=\GL_N$, an algebraic
group scheme over $\k$. We may assume the field $\F$ has positive characteristic $p$.
The map $\gr A\to \hull_\nabla \gr A$ is still $p$-power surjective by \cite[Theorem 29, Proposition 41]{FvdK}.
Write $\hull_\nabla \gr A$ as a quotient of an algebra $\k\otimes_{\F_p}R$, where $R$ is a finitely generated
$\F_p$-algebra with good filtration for $G_{\F_p}$, for instance by taking
for $\k\otimes_{\F_p}R$ the algebra $C^D$ in Lemma \ref{cover}. We
may choose $r$ so that $\gr A$ is a noetherian $\k\otimes_{\F_p}R^{(r)}$-module.
Then by Friedlander and Suslin, whose theorem \cite[Theorem 1.5, Remark 1.5.1]{Friedlander-Suslin} already has the
proper generality, we now know that $H^*(G_r,\gr A)^{(-r)}$ is a noetherian module over the graded algebra
$\bigotimes_{i=1}^rS_\k^*((\gl_n)^\#(2p^{i-1}))\otimes_{\F_p}R$. This graded algebra has good  filtration.
So our Theorem \ref{maingood} tells there are  only finitely many nonzero $H^i(G/G_r,H^*(G_r,\gr A))$ and they are 
all noetherian over $H^0(G/G_r,H^*(G_r,\gr A))$ by Corollary \ref{noetheriancohomology}. 
In view of \cite{FvdK} the proof of Touz\'e in \cite{TvdK} goes through.\qed

\begin{Remark}
 Let $G$ be a  flat affine algebraic group scheme over a ring $R$. Suppose $G_S$ satisfies (CFG) for some 
faithfully flat commutative $R$-algebra $S$. Then so does $G$. Therefore Theorem \ref{CFG:theorem} has 
consequences for some twisted families.
\end{Remark}

\section*{Reductive group schemes over a noetherian base ring.}
Let $\k$ be a noetherian ring and let $G$ be a reductive algebraic group scheme over $\Spec(\k)$, in the sense of SGA3, as always.
By \cite[Expos\'e XXII, Corollaire 2.3]{SGA3} the group scheme $G$ is locally split in the \'etale topology on $\Spec(\k)$.
Almost all the properties  we try to establish are fpqc local on $\Spec(\k)$, so in proofs we may and shall further 
assume $G$ is split. Note that we have only defined the Grosshans filtration when the group is split.

In view of what we just did for the case that $\k$ contains a field, we may
 replace $\Z$ with any noetherian ring $\k$ in section 6 of \cite{FvdK}.
Assume as always that the commutative algebra $A$ is finitely generated over the noetherian ring $\k$, 
with rational action on $A$ of $G$. When $G$ is split, we provide $A$ with the Grosshans filtration.
Further, let $M$ be a noetherian $A$-module with compatible $G$-action.
An abelian group $L$ has \emph{bounded torsion} if there is an $n\geq1$ with $nL_\tors=0$.
Summarizing section 6 of \cite{FvdK} we get
\begin{Theorem}[Provisional CFG]\label{bounded power fg:theorem}
We have
\begin{itemize}
\item Every $\h^m(G,M)$ is a noetherian $A^G$-module.
\item If $\h^*(G,A)$ is a finitely generated $\k$-algebra, then $\h^*(G,M)$ is a noetherian $\h^*(G,A)$-module.
\item If $G$ is split, then $\h^*(G,\gr A)$ is a finitely generated $\k$-algebra.
\item $\h^*(G,A)$ has bounded torsion if and only if it is a finitely generated $\k$-algebra. 
\item If\/ $\h^*(G,A)$ has bounded torsion, then the reduction $\h^\even(G,A)\to \h^\even(G,A/pA)$
is power-surjective for every prime number $p$.
\item If $ \h^\even(G,A/pA)$
is a noetherian $\h^\even(G,A)$-module for every prime number $p$, then $\h^*(G, A)$ is a finitely generated $\k$-algebra.
\end{itemize}\qed
\end{Theorem}

\begin{Remark}
 If $\k$ contains $\Q$ or  a finite ring then $\h^*(G,A)$ obviously has bounded torsion.
Also, if $H^i(G,A)$ vanishes for $i\gg0$ then $\h^*(G,A)$ has bounded torsion.
\end{Remark}

\begin{Remark}
The $\F_p$ vector space $H^1(\GGa,\F_p)$ is infinite dimensional, so  
 the hypothesis that $G$ is a reductive group scheme can not be deleted.
\end{Remark}

\end{document}